\documentclass[makeidx]{amsart}
\usepackage[active]{srcltx}
\usepackage{enumerate}
\usepackage{epsfig}
\usepackage{hyperref}
\usepackage{amsmath}
\usepackage{amssymb}
\usepackage[usenames,dvipsnames]{color}
\newtheorem{Proposition}{Proposition}

  \newtheorem{Lemma}[Proposition]{Lemma}
  
  \newtheorem{Theorem}[Proposition]{Theorem}
 
 \newtheorem{Definition}[Proposition]{Definition}
\newtheorem*{Proposition*}{Proposition}

 \newtheorem{Note}[Proposition]{Note}

\newcommand {\z}{{\noindent}}
\def\bchi{\raisebox{0.1cm}{\scalebox{1.3}{$\chi$}}}

\def\XXint#1#2#3{{\setbox0=\hbox{$#1{#2#3}{\int}$}
     \vcenter{\hbox{$#2#3$}}\kern-.5\wd0}}

\def\({\left(}
\def\){\right)}
\def\ge{\geqslant}
\def\le{\leqslant}
\def\epsilon{\varepsilon}
\def\DD{\mathbb{D}}
 
\def\CC{\mathbb{C}}
 \def\RR{\mathbb{R}}
 \def\NN{\mathbb{N}}
\def\ZZ{\mathbb{Z}}
\def\QQ{\mathbb{Q}}
\def\Re{\mathrm{Re\,\,}}

\makeindex

\title{Foundational aspects of singular integrals} 
\author {Ovidiu Costin}
\address{Mathematics Department\\The Ohio State University\\231 w 18th Ave\\Columbus 43210\\costin@math.ohio-state.edu, corresponding author.}
\author {Harvey M. Friedman}
\address{Distinguished University Professor of Mathematics, Philosophy, and  
Computer Science, The Ohio State University, Emeritus. Mathematics Department\\The Ohio State University\\231 w 18th Ave\\Columbus 43210\\friedman@math.ohio-state.edu}
\begin{document}
\maketitle
\begin{abstract} We investigate  integration of classes of real-valued continuous functions on (0,1]. Of course difficulties arise if there is a non-$L^1$ element in the class, and the Hadamard finite part integral ({\em p.f.}) does not apply. Such singular integrals arise naturally in many contexts including PDEs and singular ODEs.

The Lebesgue  integral as well as   $p.f.$, starting at zero, obey  two fundamental conditions: (i) they act as  antiderivatives and, (ii) if $f =g$  on $(0,a)$, then their integrals  from $0$ to $x$ coincide for any $x\in  (0,a)$.

We find that integrals from zero with  the essential properties of $p.f.$, plus positivity, exist by virtue of the Axiom of Choice (AC) on all functions on $(0,1]$ which are $L^1((\epsilon,1])$ for all $\epsilon>0$. However, this existence proof does not provide a satisfactory construction. Without some regularity at $0$, the existence of general antiderivatives which satisfy only (i) and (ii) above on classes with a non-$L^1$ element is independent of ZF 
(the usual ZFC axioms for mathematics without
AC), and even of ZFDC (ZF with the Axiom of  
Dependent
Choice).  Moreover we show that there is no mathematical
description that can be proved (within ZFC or even extensions of
ZFC with large cardinal hypotheses) to uniquely define such an antiderivative operator.

Such results are precisely  
formulated for a variety of sets of functions, and
proved using methods from mathematical logic, descriptive set theory  
and analysis. We also analyze $p.f.$ on analytic functions in the punctured unit disk, and make the connection to singular initial value problems.


\smallskip
\noindent \textbf{MSC classification numbers:}  32A55, 03E15, 03E25, 03E35, 03E75.
\end{abstract}

\section{Introduction}
In this paper we investigate integration of classes of real-valued continuous functions on (0,1] and of analytic functions in the punctured unit disk.

Integrals of functions which are singular in the interior of the interval of integration are  relatively well understood.  A notable example is the Hilbert transform $(Hf)(x)=\pi^{-1}\int_{-\infty}^\infty f(s)(s-x)^{-1}ds$ where the integrand is $L^1$ for large $s$. Evidently, the integrand is in $L^1(\RR)$ iff $f$ vanishes identically  and in general $Hf$  needs an interpretation.   Defined as a Cauchy principal value integral, the domain of $H$ is a set of H\"older continuous functions \cite{Musk}. However, by a substantial argument, regularity is {\em not needed}  to ensure that a natural {\em extension}  of  $Hf$  exist: $H$  is a bounded  operator on $L^p$ for any $p\in (1,\infty)$. By a theorem of Titchmarsh, for $f$ in $L^p$ as above, $H f$ exists pointwise everywhere  \cite{Titchmarsh}. $H$ also extends as a bounded operator from $L^1$ into weak-$L^1$ \cite{SteinWeiss}. The extension preserves the essential properties of $H$.

In contrast, only sufficient conditions are known for the existence one-sided singular integrals such as 
\begin{equation}\label{eq:defH2}
  \Gamma(\alpha)J^\alpha: =\int_0^x s^{\alpha-1} f(s)ds
\end{equation}
with $f$ bounded and $\Re\alpha<0$. These are frequently encountered   in PDEs, in the
analysis of differential and pseudodifferential operators, orthogonal
polynomials and many other contexts. Although the history of \eqref{eq:defH2} goes back to Liouville \cite{Liouville} and Riemann \cite{Riemann}, its first systematic treatment is due to Hadamard who interpreted integrals of the type $\int_{a}^bf(s)(b-s)^{\alpha-1}ds$ arising in hyperbolic PDEs.  In \eqref{eq:defH2}, sufficient smoothness of $f$,
\begin{equation}
  \label{eq:sufc2}
  \text{$f\in C^{n}((a,b])$ and  $f^{(n)}(s)s^{\alpha+n-1}\in L^1$}
\end{equation}
 allows for  the Hadamard ``partie finie'' ($p.f.$,  finite part, see. \S\ref{Appendix}) to provide an extension of the usual integral. This extension, as shown  by Hadamard has the properties (i) and (ii) above. In fact, it  has all the properties of Lebesgue integration but positivity \cite{Hadamard}. Riesz \cite{Riesz} showed that  $p.f.$  can be (essentially)  equivalently defined by analytic continuation starting with $\Re\alpha>0$; Schwartz (see \cite{Schwartz} and \cite{Hormander}) reinterprets $p.f.$ as a distributional regularization. See the Appendix for a brief review and further references. These reinterpretations make sense if \eqref{eq:sufc2} holds. In this paper, we use {\em Riesz's definition of p.f.}

Related questions arise in singular ODEs, when a non- $L^1$ fundamental solution  needs to be defined in terms of properties of solutions at the singularity. This is possible for meromorphic ODEs if the order of the poles is sufficiently low \cite{CL} or under other regularity conditions. One of the  weakest such regularity condition is \'Ecalle's {\em analyzability}, see e.g. \cite{Ecalle1,Ecalle2} and \cite{Book} and references therein.

If the {\em integrand} in \eqref{eq:defH2} is in $L^1$ then $p.f.$ is just the Lebesgue integral. Evidently, imposing  the condition that a set of functions has a non-$L^1$ element is {\em weaker} than assuming that p.f. is not applicable to the whole set.

In this paper we show that regularity conditions  {\em are needed} for integrals such as \eqref{eq:defH2} to make sense in sufficiently rich classes of functions with a non-$L^1$ element.

\section{Setting}\label{SSetting}

In the following, ZFC is the standard axiomatization for mathematics. ZF is
ZFC without the axiom of choice. ZFDC is ZF extended with a
weak form of the axiom of choice called Dependent Choice,
abbreviated as DC. In the following, $\NN=\ZZ^+\cup\{0\}$.
\footnote{DC is a weak baseline form of the axiom of choice. DC is
used to prove the existence of infinite  sequences, and is formulated as follows.
Let $R$ be a binary relation (set of ordered pairs), where
$(\forall x \in A)(\exists y \in A)(R(x,y))$. For each $x \in A$, there exists an
infinite sequence $x = (x_0,x_1,...)$ such that $x_0=x$ and for all $i \in N$,
$R(x_i,x_{i+1})$. It has been shown that DC is provably equivalent,
over ZF, to the Baire category theorem for complete metric
spaces. See \cite{Charles}.}

To analyze the existence of an inverse of differentiation $\mathcal{P}_0$ at a singular point, say zero, we work with functions defined to the right of zero, that is $f:(0,a]\to \RR$ for some $a>0$ which may depend on $f$. In most interesting cases, $0$ will be a singular point.  To avoid unwanted obstructions due to other singularities, we assume  $f\in C^k( (0,a])$  for some $k\in\NN\cup \{\infty\}$. 
One  may smoothly extend such a function to $(0,1]$, and thus we will {\em assume} that all our functions are in $C((0,1])$. Integration of extensions or restrictions of functions can be trivially obtained from the ones defined on functions with a common domain.

\begin{Definition} {\rm We define the germ equivalence relation,}
  \begin{equation}\label{equivrel}
  f\sim_0 g\ \text{if and only if } \exists a>0 \text{ s.t. } f=g \text{ on } (0,a]
\end{equation}
\end{Definition}
\begin{Definition}{\rm 
An operator $A$ is based at $0^+$ if
\begin{equation}
  \label{eq:b0}
   f\sim_0 g \Rightarrow A(f) \sim_0 A(g) 
\end{equation}
Operators based at other points in $\RR\cup\{\infty\}$ are analogously defined.
}\end{Definition}

\subsection{Function sets}\label{spaces} 
The classical normed spaces imposing size but not smoothness constraints, such as weighted $L^p$ space and  Orlicz spaces  have the so called {\em lattice property}: if $g$ is measurable, $|g|\le |f|$,  and  $f$ is in the space, then $g$ is in the space.

We work in more general sets, with  possible growth restrictions-- through bounds-- for instance as in the classical spaces above,  while trying to prevent regularity at zero. \footnote{\z There is no all-encompassing definition of  regularity we are aware of. Certainly the integrand in \eqref{eq:defH2}  as a whole is not smooth in any usual sense; its {\em reciprocal} is smooth. An interesting  discussion of  smoothness and  links to approximability by analytic functions, is in \cite{Dynkin}.}. We work in structures of the type \eqref{eq:defspg} below, which  allow for arbitrarily fast oscillations at zero; these seem to exclude  classical regularity of any kind. The following is a weakening of the lattice property.
\begin{Definition}{\rm 
For a {\bf continuous}  function $f$ on $(0,1]$ we denote by $f^{\Join}$ the multiplicative semigroup generated by $f$ and all smooth functions bounded by $1$ on $(0,1]$:  
\begin{equation}
  \label{eq:defspg}
  f^{\Join}=\{hf:h\in C^\infty((0,1]),\,\,\|h\|_{\infty}\le 1\}
\end{equation}
\begin{Note}\label{NN3}
 In this paper we  use the notation $f^{\Join} $ {\bf only} for functions in $C((0,1])$.
\end{Note}
It can be immediately verified that if $f^{\Join}=g^{\Join}\Rightarrow (f\in L^1\Leftrightarrow g\in L^1)$. 
 }\end{Definition}

\subsection{Inverses of differentiation}
For  our negative results, we impose only
the properties of the integral from zero which, arguably, any extension should satisfy.
\begin{Definition}\label{R1}\rm For $f\in C((0,1])$, $Op(f^{\Join})$ is the collection operators   $\mathcal{P}_0$  from
$f^{\Join}$ into $C^1((0,1])$  s.t. 

(I) $\mathcal{P}_0$ is based at $0^+$ 

(II)  $\forall g\in f^{\Join}$ we have
 $(\mathcal{P}_0 g)'=g$. 
\end{Definition}
The elements of $Op(f^{\Join})$ are clearly inverses of differentiation on $f^{\Join}$. If $f\in L^1$ then $Op(f^{\Join})\ne \emptyset$, since the restriction of the Lebesgue integral from zero is in $Op(f^{\Join})$.  Note that (II) implies  $\mathcal{P}_0f'=f+const$  whenever $f'\in  C((0,1])$. Also see 
Note \ref{NN3}.
\begin{Definition}\label{Ext}{\rm 
An {\em extension of the integral from zero} is an  operator  $\mathcal{Q}_0: C((0,1]) \to C^1((0,1])$ based at $0^+$ such that $(Q_0 f)'=f$, and with the additional properties (III)  $Q_0$ is linear, (IV) $( \mathcal{Q}_0(f))(x)=\int_0^x f$ for  $f\in L^1((0,1])$, and  (V) if $f>0$ then $\exists a\text{ s.t. }(\mathcal{Q}_0f)(x)>0$ for all  $x\in (0,a)$.  

EI denotes that set of all such  $\mathcal{Q}_0$. }
\end{Definition}
\begin{Note}{\rm It will be clear that our negative results about an $Op(f^{\Join})$ are inherited by any sets (such as, say,   weighted $L^p$ spaces) containing the corresponding $f^{\Join}$}.
\end{Note}

\begin{Note}\label{nrange}{\rm  
\label{n21} \z The conditions in Definition \ref{R1} are not stronger than taking  $AC((0,1])$ instead of $C^1((0,1])$ and adding ``{\em a.e.}'' after (II). Indeed,  since  $(\int_1^x f)'=f$ and the  range of both $\mathcal{P}_0$ and  $(x,f)\mapsto \int_1^x f$  is $AC((0,1]$, for any $f$ there is a constant $C(f)$ s.t. $\forall x\in (0,1]$ we have $[\mathcal{P}_0 f](x)=\int_1^x f+C(f)$. This also implies that if $f\in C((0,1]$, $\mathcal{P}_0f \in C^1((0,1])$. Then $(\mathcal{P}_0 f)'=f$ everywhere implying that  for any $f$ in the domain of $\mathcal{P}_0$, all $ y\in(0,1]$  and  all $x\in (0,y)$ we have $[\mathcal{P}_0(f)](y)=[\mathcal{P}_0(f)](x)+\int_x^y f$ where $\int$ is the usual Lebesgue integral.}
\end{Note}
Note that Definition \ref{Ext} imposes stronger conditions than Definition \ref{R1}. That is because we are using Definition \ref{Ext} mostly for  positive results, and  Definition \ref{R1} for the negative results. 
\begin{Proposition}\label{L10}
If $f\notin L^1$  there is a $\tau\in f^{\Join}$, definable in terms of $f$, s.t 
\begin{equation}
  \label{eq:cdpm}
 \tau\ge 0 \ \text{and}\  \int_0^1 \tau=\infty 
\end{equation}
\end{Proposition}
The proof of this proposition is given in \S\ref{S43}.
\begin{Proposition} \label{P44}
\label{parti}  
For $p\in (1,\infty)$,  $L^p_w((0,1])\not\subseteq L^1$,  {\bf iff}  $\int_0^1 w^{-\frac{1}{p-1}}=\infty$. If $p=\infty$ then this conclusion holds {\bf iff} $\int_0^1(1/w)=\infty$, while if $p=1$, it holds {\bf iff} $w(0)=0$. 
A  $\tau\in C((0,1])$ such that \eqref{eq:cdpm} holds can be defined in terms of $w$.
\end{Proposition}
\z The proof of Proposition \ref{P44} is given in \S\ref{explcP}.
\begin{Note}
  Existence of a $\mathcal{P}_0$ on  a weighted $L^p$ space, $L^p_w$ is the same as existence of a $\mathcal{P}_0$ for every   $f^{\Join}, f\in L^p_w$. For negative results, this allows us to only consider  $Op(f^{\Join})$.
\end{Note}
On  $C((0,1])$ we use the topology of uniform convergence on compact sets. This topology is induced by the sequence of seminorms  $\mathcal{F}(f):=\big(\sup_{[n^{-1},1]}|f|\big)_{n\in\ZZ^+}$. The sequence vanishes iff $f=0$. Hence, the space is metrizable (see e.g. \cite{Kechris} p.3). Since the polynomials with rational coefficients are dense in this space, the space is Polish.
\begin{Definition}
  We denote by $\mathbb{P}$ the Polish space above. A set is said to be Borel measurable if it is Borel measurable in $\mathbb{P}$.
\end{Definition}
\section{Main results}\label{Results}
According to Theorem \ref{T1}  (\ref{T1a}) below, ZFC proves that $EI$ has  $2^{\frak{c}}$ elements. This is proved in \S\ref{existchoice} using
the Axiom of Choice. The rest of our results are negative and establish how pathological the construction of elements of $EI$ must be.

Specifically, according to Theorem \ref{New10},  there is no description for  
which it is
provable in ZFC that the description uniquely defines an element of $Op(f^{\Join})$ with $f\notin L^1$. In particular, this rules out usable formulas
for such operators that can be proved to work within the
usual axioms for mathematics.

We emphasize Theorem \ref{New10}  and the sharpened form Theorem \ref{New11} over Theorem \ref{T1} \eqref{T1b}, \eqref{T1c}, \eqref{T1d}. This is because the Axiom of
Choice has long since transitioned from being
controversial, to being accepted as part of the
usual ZFC axiomatization for mathematics. However,
the impossibility of giving explicit examples that
can be verified to hold in ZFC represents a deeper
and more serious impossibility than merely
requiring the use of the axiom of choice to prove
existence (beyond the relatively benign DC). In
practice, the two kinds of impossibility are
closely related, although there are counterexamples
to direct implications between the two. Theorems \ref{T2} and  Theorem  \ref{T2.5} provide the kind of if and only if information given by Theorem  \ref{T1} \eqref{T1b} -- but in the setting of ZFC.

In the following we denote  by $L^1$, by abuse of notation, the space of measurable {\em functions}   $f:(0,1]\to \RR$ for which $\lim_{\epsilon\to 0^+}\int_\epsilon^1 f$ exists.
\begin{Proposition}\label{Ck}
 The sets $C^k((0,1])\,\,, k\in\NN\cup\{\infty\}$ are Borel measurable subsets of $\mathbb{P}$. \end{Proposition}
The proof of this proposition is given in \S\ref{SCk}.
We say that a sentence is independent  of a theory if it can neither be proved or refuted in that theory.
\begin{Theorem}\label{T1}
  \begin{enumerate}[(a)]
\item \label{T1a} ZFC proves that $EI$ has  $2^{\frak{c}}$ elements. 
\item\label{T1b} 
 ZFDC proves the following: $ Op(f^{\Join})$ has a  Borel measurable element
if and only if  $f\in L^1$.
\item\label{T1c}  
ZFDC proves  that, if $f\notin L^1$ and  $Op(f^{\Join})\ne \emptyset$, then there
is a set of reals which is not Baire measurable.
\item \label{T1d}    The statement $(\exists f\notin L^1)( Op(f^{\Join})\ne \emptyset)$ is not provable in ZFDC.

\end{enumerate}\end{Theorem}
The proof of Theorem \ref{T1} \eqref{T1a} extends easily to the measurable functions on $(0,1]$ for which  the Lebesgue integral from 1 exists for any $\epsilon>0$ (the limit as $\epsilon\to 0^+$ of the integral may not exist). 

Also, note the important if and only if nature of Theorem \ref{T1}\eqref{T1b}.

The following is an easy corollary of Theorem~\ref{T1}.
\begin{Theorem}\label{T14}
         $EI\neq\emptyset$   is  independent of ZFDC.  
For any $ k\in\NN\cup \{\infty\}$, $(\exists f\in C^k((0,1])\setminus L^1)( Op(f^{\Join})\ne \emptyset)$ is also independent of ZFDC. 
\end{Theorem}

\begin{Theorem}\label{New10}
 There is no definition which, provably in ZFC,
uniquely defines  some element of some $Op(f^{\Join})$ with $ f\notin L^1$. This also holds for ZFC
extended by the usual large cardinal hypotheses.
\end{Theorem}
  \begin{Theorem}\label{New11}
There is no definition which, provably in ZFC, uniquely defines a
  function whose domain is a set of real numbers with a value that is in $Op(f^{\Join})$  for some $ f\notin L^1$. This also holds for ZFC extended by
the usual large cardinal hypotheses.
      \end{Theorem}
     
In
order to obtain  if and only if information as in Theorem \ref{T1} \eqref{T1b} in the
context of ZFC, we need to work with concretely given $f\in C((0,1])$. We say that $E\subset  \ZZ^4$  codes  an $f \in C((0,1])$ if and only if $\forall (a,b,c,d) \in \ZZ^4$,  $a/b < f(c/d)$ iff $(a,b,c,d) \in  E$. An arithmetic presentation of $E \subset  \ZZ^4$ takes the form $\{(a,b,c,d) \in \ZZ^4: \phi\}$, where $\phi$ is a formula involving $\forall, \exists, \neg, \vee , \wedge, +,-,\cdot,<,0,1,$  variables ranging over $\ZZ$,  with at most the free variables $a,b,c,d$.

 \begin{Theorem}\label{T2}
Let $E\subset \ZZ^4$ be arithmetically presented, where ZFC proves that $E$ codes some  $f\in C((0,1])$.  There is a definition, which, provably in ZFC, uniquely defines an element of $Op(f^{\Join})$ if and only if ZFC proves $f\in L^1$. This also holds for ZFC extended by the usual
large cardinal hypotheses.
  \end{Theorem}
The following result strengthens the forward direction in
Theorem \ref{T2}.
\begin{Theorem}\label{T2.5} 
 Let $E\subset \ZZ^4$ be arithmetically presented, where ZFC proves that $E$ codes some  $f\in C((0,1])$. There is a definition which, provably in ZFC,
uniquely defines a function whose domain is a set of real numbers, with a value in  $Op(f^{\Join})$  if and only if ZFC proves $f\in L^1$. This also holds for
ZFC extended by the usual large cardinal hypotheses.
\end{Theorem}
Examples of arithmetically presented functions include limits of pointwise convergent
sequences of rational polynomials, provided the sequence is algorithmically
computable. In addition, for algorithmically computable double sequences
of rational polynomials, the pointwise limit of pointwise limits is arithmetically presented, provided we have the requisite
pointwise convergence. Also, compositions of arithmetically presented functions are arithmetically presented. Of course, such functions may or may not be continuous. Standard elementary and special functions over the rationals are arithmetically presented.

Note that Theorem \ref{New11} rules out any description, even using unspecified real number parameters, for which it is provable in ZFC that for some choice of these parameters, the description uniquely defines an element of $Op(f^{\Join}),  f\notin L^1$. This rules out usable formulas
for such operators that can be proved to work
within the usual ZFC axioms for mathematics.

The proofs of Theorems \ref{T1} \eqref{T1b}, \eqref{T1c}  in \S\ref{S72}  rely on the
Interface Theorems from \S\ref{Intfl}.  The Interface Theorems
show how to go explicitly from any element of $ Op(f^{\Join}), f\notin L^1 $ to a
corresponding summation operator which maps $\{0,1\}^{\NN}$ into $\RR^\NN$.
From the point of view of descriptive set theory and
mathematical logic, it is easier to work with summation
operators. In \S\ref{sequences} we establish the
results about summation operators that we use in \S\ref{S72}, using standard techniques from descriptive set theory.

Theorem \ref{T1} \eqref{T1d}  is proved in \S\ref{S72} and Theorem \ref{T2.5} is proved in \S\ref{S6.6-}. Theorem \ref{T2} is an immediate consequence of Theorem \ref{T2.5}. Theorems \ref{T1} \eqref{T1d} and  \ref{T2.5} 
follow from Theorem \ref{T1} \eqref{T1c} using well known results
from mathematical logic.

The key to obtaining these negative mathematical logic results from the analytic questions is essentially the content of \S\ref{Intfl}.

\subsection{Links to other problems} \label{IVP}

\begin{Note}\label{NN17}
  {\rm Also, our results appear to preclude the existence of an integration operator over a sufficiently large class of functions defined on {\bf No}, the surreal numbers of J.H. Conway, \cite{Conway}. Indeed, if, say continuous functions extended past the gap at $\infty$ and an integral existed for them, then $\int_x^\infty f:=F\int_x^\omega f$ where $F$ is the finite part of a surreal number, would violate the conclusions of our theorems, since the proofs of existence of surreal objects are done using ``earliness'' and never  use  AC. 

The results of Note \ref{NN17} will appear elsewhere. 
}\end{Note}

\begin{Note}
  {\rm For the link with $p.f.$ see the Appendix.
}\end{Note}

\begin{Note}[Remarks about singular initial value problems]
{\rm 
   In the realm of ODEs with conditions placed at a singularity,  arguably the simplest  example is the linear ODE $f'=gf$   where $g$  is singular at zero. Can  conditions {\em at zero}  separate solutions?  In case $g$ is singular at zero, but $g\in L^1((0,1])$, then the answer is yes: the general solution is  $y=(x\mapsto C+\int_0^x g)$ and $y(0)=C$ is such a condition.  However, in the case $g\notin L^1$, the question reduces to the form considered in \S\ref{SSetting}--by elementary operations. Many other ODEs can brought to our setting. Such  questions will be treated elsewhere.
}\end{Note}

The Hadamard $p.f.$ in the complex domain is analyzed in the Appendix. The results are similar to the ones on $(0,1]$ and we will only prove two main ones.

\section{Interface theorems}\label{Intfl}
The precise notions of summation operators at infinity are given in Definitions \ref{defsum} and \ref{atinfty} below.
\subsection{Informal description} 

Consider  the Cantor set
\begin{equation}
  \label{eq:cant3}
\{0,1\}^\NN:=\{(a_i)_{i\in\NN}:a_i\in \{0,1\}\} 
\end{equation}
  For each of the sets $f^{\Join}$ and any nontrivial extension of integration we define  a summation operator {\em from $n$ to infinity} (based at infinity, in the sense of Definition \ref{atinfty} below,  on $\{0,1\}^\NN$  with values in $\RR^{\NN}$. Informally, this is  a finite-valued summation operator  with the property that for any two sequences which coincide eventually the sum also coincides eventually (see Proposition \ref{providesum}): \begin{equation}
  \label{eq:sumop}
(\exists \,N)(\forall n\ge N)( \ a_n=a'_n)\Rightarrow  (\exists N)(\forall n\ge N)\(\sum_{i=n}^\infty a_i=\sum_{i=n}^\infty a'_i\)
\end{equation}
Implausible as they might seem, such  operators exist assuming AC. They are a byproduct of extensions of $p.f.$ (e.g. to the whole of $C^{\infty}((0,1])$ with {\em no growth or regularity condition} at zero), which also exist assuming  AC.   As expected, such a  summation  is pathological and no formula can exist for it.

To formulate negative results about the summation operator and for proving them, descriptive set theory and mathematical logic are used non-trivially.
\subsection{Detailed  results}\label{details2}  
\begin{Note}
 {\rm A summation operator acting on the sequence $x_n$ is naturally defined as a solution of the recurrence $s_{n+1}-s_n=x_n$. Clearly two solutions differ by a  constant. This motivates the following.} 
\end{Note}
\begin{Definition}\label{defsum}
   Let $ x \in \{0,1\}^\NN$. The standard summation for $x$,
written $\Sigma(x)$, is $(x_0,x_0+x_1,x_0+x_1+x_2,...)$. $\mathcal{S}$ is a summation
operator if and only if $\mathcal{S}:\{0,1\}^\NN \to \RR^\NN$, where for all $x \in
\{0,1\}^\NN$ there exists $c \in \RR$ such that $\mathcal{S}(x) = \Sigma(x) + c$. I.e.,
$\mathcal{S}(x)$ is $\Sigma(x$) with $c$ added to all terms of $\Sigma(x)$.
\end{Definition}
\begin{Definition}\label{atinfty}
   For any set $X$, $X^{\NN}$ is the set of all $f:N \to X$,
which is the same as the set of all infinite sequences from
$X$ indexed from $0$. For $x,y \in X^{\NN}$, $x \sim_{\infty} y$ if and only if
$(\exists n)(\forall m \ge n)(x_m = y_m)$. $F:X^{\NN} \to Y^{\NN}$ is based at infinity if and
only if for all $x,y \in X^{\NN}$, $x \sim_{\infty} y \Rightarrow F(x) \sim_{\infty} F(y)$. We use $[x]_{\infty}$
for $\{y \in X^{\NN}: x \sim_{\infty} y\}$.
\end{Definition}
\begin{Definition}
   $ Su $ is the collection of all summation
operators on $\{0,1\}^\NN$ based at infinity.

\end{Definition}

The following two Propositions follow easily from Proposition \ref{providesum} and its proof.
\begin{Proposition}\label{PZF}
  ZF proves the following: If $f\notin L^1\text{ and }Op(f^{\Join})\ne\emptyset$ then $ Su \ne\emptyset$.
\end{Proposition}
\begin{Proposition}\label{PZFDC}
ZFDC \footnote{DC is used to prove basic facts about  Borel measurability.}  proves the following: if there is a Borel measurable  element of some $ Op(f^{\Join}), f\notin L^1$ then there is a Borel  measurable  element of $ Su $.
\end{Proposition}

\subsection{Proof of Proposition \ref{L10}} \label{S43}
\begin{proof} 
A similar construction will be used in the proof of Proposition \ref{providesum}. By standard measure theory, if $f\notin L^1$, then $\int f^+$ or $\int f^-$ is $+\infty$ where $f^\pm$ are the positive/negative parts of $f$.  Assume  $\int f^+=+\infty$ (if $\int f^-=+\infty$, the construction is essentially the same, with minor modification such as $h\leftrightarrow -h$.). Define $G=\{x\in (0,1]:f(x)\ge \frac12$ and   $L=\{x\in (0,1]:f(x)\le  \frac14$\}; they are clearly both closed in the relative topology on $(0,1]$.  Let $c_0=\max G$ and $b_0= \max\{x\in L:x\le c_0\}$. Inductively, let $c_j=\max\{x\in G:x\le b_{j-1}\}$ and $b_j=\max\{x\in L:x\le c_j\}$. Also inductively, let $\epsilon_j<\min\{\frac12(c_j-b_j),\frac12(b_j-c_{j+1}))\}$ and define $h\in C^\infty((0,1])$ be a gluing function s.t. $h=1$ on $[b_j,c_j]$ and $h=0$ on $[c_{j+1}+\epsilon, b_j-\epsilon]$. It is clear that $\tau:=hf\ge 0$ and $\int_0^1 (f^+-\tau)<1$ where $\int$ is the usual Lebesgue integral.  
\end{proof}
\subsection{Proof of Proposition \ref{P44}} \label{explcP}
\begin{proof} If $p=\infty$ then clearly $w|\tau|\le 1$ implies $\int |\tau|=\int w|\tau| w^{-1}\le \int_0^1 w^{-1}<\infty $ if   $\int_0^1 (1/w)<\infty$. If, instead,  $\int_0^1 (1/w)=\infty$ then clearly $\tau:=1/w\in L^\infty_w$ while $\int\tau=\infty.$  An  explicit $\tau$ with $\|\tau\|_1=\infty$ is  $w^{-1}$.

If $p=1$ then  we assume, without loss of generality, that $w\le 1$. If $w(0)\ne 0$, then $w, 1/w$ are bounded, $L^1_w=L^1$,  and  there is nothing to prove. Otherwise, let $A_k=\{x:w(x)\in [1/k^2,1/(k+1)^2)\}$ and $\tau=\sum \bchi_{A_k}/m(A_k)$ where $m$ is the usual Lebesgue measure. Clearly,  the $A_k$  are disjoint and 
$\int_{[0,1]} \tau w\le \sum_k{k^{-2}}<\infty$ while $\int_0^1 \tau =\sum_{k=1}^\infty 1=\infty$. This $\tau$ can be made smooth as in Lemma~\ref{L10}.

Let $p\in (1,\infty)$ and assume $\int_0^1 w^{-\frac{p}{p-1}}<\infty$ and let $\tau\ge 0\in L^p_w$. Then, by H\"older's inequality we have
  \begin{equation}
    \label{eq:holder}
    \int_0^1\tau= \int_0^1[w \tau]w^{-1}\le \(\int_0^1(w\tau)^p\)^{\frac1p} \(\int_0^1w^{-\frac{p}{p-1}}\)^{\frac1q}<\infty
  \end{equation}
Conversely, assume that $\int_0^1 w^{-\frac{p}{p-1}}=\infty$. Straightforward calculations show that $\tau(x)=w(x)^{-\frac{p}{p-1}}/\int_x^1w^{-\frac{p}{p-1}}$ satisfies our requirements.
\end{proof}

\subsection{Proof of Proposition \ref{Ck}}\label{SCk} $J_1$, integration from $1$, is continuous and injective on $\mathbb{P}$.  Hence, by \cite{Kechris} \footnote{See Corollary 15.2, p.89.}, for any $ \mathcal{O}\subset \mathbb{P}$ open, $J_1(\mathcal{O})$ is Borel measurable. Let $D:C^1((0,1])\mapsto C((0,1])$ be the usual differentiation operator. We claim that $D$  is Borel measurable. Let $\mathcal{O}\in\mathbb{P}$ be open. Then, by elementary calculus, $f\in D^{-1}(\mathcal{O})$ iff $f\in H^{-1}(J_1(\mathcal{O}))$ where $H:=f\mapsto f-f(1)$  is  continuous on $\mathbb{P}$. By calculus, $\forall k\in\NN$, we have  $C^k((0,1])=(D^k)^{-1}(\mathbb{P})$. Hence $\forall k\in \NN$, $C^k((0,1])$  (and  $C^\infty((0,1])=\cap_{k\in\NN} C^k((0,1])$) are  Borel measurable. \footnote{In fact, $C^k$, $k\in\NN\cup\{\infty\}$  is  $\boldsymbol{\Pi}^0_3$-complete, see  \cite{Marcone} and \cite{Kechris}, \S 23 D.} $\qed$.

\begin{Lemma}\label{Cantor2}
Given any $f^{\Join}, f\notin L^1$, a decreasing sequence $(\alpha_k)_{k\in\NN}$ in $(0,1]$  can be defined in terms of $f$ such that
\begin{equation}
  \label{eq:intab}
  \int_{\alpha_k}^{\alpha_{k+1}}\tau=1
\end{equation}
where $\tau$ is constructed as in Proposition \ref{L10}.
\end{Lemma}
\begin{proof}
 The function $\tau$ constructed in Proposition \ref{L10} has the property that $\theta(x)=\int_x^1\tau$ is strictly decreasing, thus continuously invertible from $\RR^+$ into $\RR^+$. We let $\alpha_0 =1$,  $\alpha_k=\theta^{-1}(k), k\ge 1$ and note that \eqref{eq:intab} holds.
\end{proof}

\begin{Proposition}\label{providesum}
There is an explicit procedure for going from any element of $Op(f^{\Join}),f\notin L^1$ to an element of $Su$. Hence Propositions \ref{PZF} and \ref{PZFDC} hold.
\end{Proposition}
\begin{proof}
  For $\tau \in T$ s.t. \eqref{eq:cdpm} holds, consider the open set 
$$\mathcal{O}_a=\bigcup_{n:a_n=0} (\alpha_{n+1},\alpha_n)$$
With  $\alpha_k$ as in  \eqref{eq:intab} we construct a $C^\infty$ function $h$ out of  $\tau$ from which is $1$ on $\mathcal{O}_a^c$ and is zeroed out as in Lemma \ref{L10} with the role of $(b_{n+1},c_n)$ played by $(\alpha_{n+1},\alpha_n)$ dividing now $\epsilon_n$ by $2\max_{(\alpha_{n+1},\alpha_n)}\tau$ \footnote{Nonzero by \eqref{eq:intab}.}. This $h$  has the property that the Lebesgue integral $\int_{0}^1 |\bchi_{\mathcal{O}^c_a}\tau-\tau h|<\frac12$. Define 
  \begin{equation}\label{defsigma}
 x_{k;a}=\left[\mathcal{P}_0 (\tau h)\right] (\alpha_k)+\int_0^{\alpha_k}(\bchi_{\mathcal{O}_a^c}\tau-\tau h)
  \end{equation}
where the last integral is the Lebesgue integral.  By construction, Note \ref{nrange}, \eqref{n21} and \eqref{defsigma}, we have $x_{n+1;a}-x_{n;a}=\int_{\alpha_{n+1}}^{\alpha_n}\tau\bchi_{\mathcal{O}_a^c}$=$\int_{\alpha_{n+1}}^{\alpha_n}\tau=a_n$. Thus 
$$\mathcal{S}=a\mapsto (x_{0;a},x_{0;a}+x_{1;a},x_{0;a}+x_{1;a}+x_{2;a},...)$$
 is a summation operator on $\{0,1\}^{\NN}$. Proposition \ref{PZF} follows since all constructions have been done in $ZF$. Proposition \ref{PZFDC} follows from the fact that the maps used in the constructions in this proof are manifestly Borel measurable.
\end{proof}

\begin{Note}\rm{
  We remark that all the proofs in \S\ref{Intfl} are carried out in ZF. 
}\end{Note}
\section{Summation operators}\label{sequences}
Until the proof of Theorem \ref{THEOREM 7.} is complete, we fix a
summation operator $ \mathcal{S}:\{0,1\}^{\NN} \to \RR^\NN$, see Definition 
\ref{defsum}, based at infinity, and
prove that $\mathcal{S}$ is not Baire measurable. We assume that $\mathcal{S}$ is
Baire measurable, and obtain a contradiction. The proof
takes place within ZFDC, and is an application of a widely
used technique from descriptive set theory.
For useful information about Baire spaces and Baire
category, we refer the reader to Kechris, \cite{Kechris}, section 8.
\begin{Definition}
  Let $f:X \to Y$, where $X,Y$ are topological
spaces, and $E\subseteq X$. We say that $f$ is continuous over $E$
if and only if  $f$  restricted to  $E$ is a continuous
function where  $E$ is given the subspace (i.e., induced) topology.
\end{Definition}
\begin{Lemma}[\cite{Kechris}, 8.38 p. 52]\label{L1.1}
 Let $X$ be a Baire space and $Y$ be a second countable space and assume $f:X\to Y$ is  Baire measurable. Then $f$ is continuous over a comeager subset of $X$.
\end{Lemma}
\begin{Lemma}\label{L1.2}
  Let $f:X \to X$  be a bicontinuous bijection,
where $X$ is a Baire space. If $E\subseteq X$   is comeager then $f^{-1}(E)$ is comeager.
\end{Lemma}
\begin{proof}
  It suffices to observe that the forward image of
any dense open set under $f$ is a dense open set.
\end{proof}
\begin{Lemma}\label{L35}
   Let $E \subseteq \{0,1\}^{\NN}$ be comeager in$ \{0,1\}^{\NN}$. $\{x \in \{0,1\}^{\NN}:
[x]_\infty \subseteq E\}$ is comeager in $\{0,1\}^{\NN}$.
\end{Lemma}
\begin{proof}
  We apply Lemma \ref{L1.2} to the Baire space $ \{0,1\}^{\NN}$. For each
nonempty finite sequence $\alpha$ from $\{0,1\}$, let $\alpha^* \in \{0,1\}^{\NN}$ be $\alpha$
extended with all $0$'s, and $f_{\alpha}$ be the bicontinuous bijection
of $\{0,1\}^{\NN}$ given by $f_{\alpha}(x) = x + \alpha^*$. Here $+$ is addition modulo $2$. Obviously $\{x \in \{0,1\}^{\NN}:
[x]_\infty \subseteq E\} = \cap_{\alpha} f_{\alpha}^{-1}[E]$, which by Lemma \ref{L1.2}, is the countable
intersection of sets comeager in $\{0,1\}^{\NN}$.
\end{proof}
\begin{Lemma}\label{L36}
   Let $F:\{0,1\}^{\NN} \to \RR$ be Baire measurable. There
exists $x \in \{0,1\}^{\NN}$ and a finite initial segment $\alpha$ of $x$ such
that $(\forall y \in [x]_{\infty} \cap \{0,1\}^{\NN})(y \text{ \rm extends }\alpha \Rightarrow |F(x) - F(y)| < 1)$.
\end{Lemma}
\begin{proof}
  By Lemma \ref{L1.1}, $F$ is continuous over a comeager set $E \subseteq
\{0,1\}^{\NN}$. By Lemma \ref{L35}, fix $[x]_{\infty} \subseteq E$, and let $F(x) = c \in \RR$. $F^{-1}
[(c-\tfrac12,c+\tfrac12]$ is an open subset of $E$ (as a subspace of
$\{0,1\}^{\NN}$) that contains $x$. This open subset of $E$ must contain
all elements of $[x]_{\infty} \cap \{0,1\}^{\NN}$ that extend some particular
finite initial segment $\alpha$ of $x$.
\end{proof}
\begin{Definition}
  $\mathcal{S}^{*}:\{0,1\}^{\NN} \to \RR^\NN$ is defined by $\mathcal{S}^{*}(x) = \mathcal{S}(x) -
\Sigma(x)$, which must be an element of $\RR^\NN$ whose terms are all the
same. $\mathcal{S}^{**}(x)$ is the unique term of $\mathcal{S}^{*}(x)$.
\end{Definition}
\begin{Lemma}
   $\mathcal{S}^{*}$ and $\mathcal{S}^{**}$ are Baire measurable.
\end{Lemma}
\begin{proof}
  We first show that $\mathcal{S}^{*}$ is Baire measurable. Let
$J:\{0,1\}^{\NN} \to \RR^\NN \times \RR^\NN$ be given by $J(x) = (\mathcal{S}(x),\Sigma(x))$. Then $\mathcal{S}^{*}$
is
the composition of $J$ with subtraction; i.e., to evaluate
$\mathcal{S}^{*}(x)$, first apply $J$, and then apply subtraction. Let $V \subseteq \RR^\NN$
be open. Then $(\mathcal{S}^{*})^{-1}[V] = J^{-1}[W]$, where $W \subseteq \RR^\NN \times \RR^\NN$ is the
inverse image of subtraction on $V$. By the continuity of
subtraction, $W$ is open. Now $W$ is a countable union of
finite intersections of Cartesian products of open subsets
of $\RR^\NN$. Note that the inverse image of $J$ on the Cartesian
product of any two open subsets of $\RR^\NN$ is Baire measurable.
Hence the inverse image of $J$ on any open subset of $\RR^\NN \times \RR^\NN$
is Baire measurable, as required. To see that $\mathcal{S}^{**}$ is Baire
measurable, note that $\mathcal{S}^{**}$ is the composition of $\mathcal{S}^{*}$ with the
first projection function $\pi_1$; i.e., to evaluate $\mathcal{S}^{**}(x)$,
first apply $\mathcal{S}^{**}$ and then apply $\pi_1$. Use the continuity of $\pi_1$.
\end{proof}
\begin{Theorem}\label{THEOREM 7.}
The following is provable in ZFDC.   There is no Baire measurable summation operator
$\mathcal{S}:\{0,1\}^\NN \to \RR^\NN$ based at infinity.
\end{Theorem}
\begin{proof}
   We have only to complete the promised contradiction.
Since $\mathcal{S}^{**}$ is Baire measurable, by Lemma \ref{L36}, fix $x \in \{0,1\}^{\NN}$
and a finite initial segment $\alpha$ of $x$ such that $(\forall y \in [x]_{\infty} \cap
\{0,1\}^{\NN})(y\text{ extends }\alpha \Rightarrow |\mathcal{S}^{**}(x) - \mathcal{S}^{**}(y)| < 1)$. Let $y\in [x]_{\infty}
\cap \{0,1\}^{\NN}$ extend $\alpha$ and agree everywhere with $x$ except at
exactly one argument (arguments are elements of $\NN$).
Obviously $|\Sigma(x) - \Sigma(y)|$ is eventually $1$ or eventually $-1$.
Since $x \sim_{\infty} y$, $\mathcal{S}(x) \sim_{\infty} \mathcal{S}(y)$, and so $\mathcal{S}(x)$ and $\mathcal{S}(y)$ eventually
agree. Now $\mathcal{S}^{*}(x) = \mathcal{S}(x) - \Sigma(x)$ and $\mathcal{S}^{*}(y) = \mathcal{S}(y) - \Sigma(y)$.
Hence $\mathcal{S}^{*}(x) - \mathcal{S}^{*}(y) = \mathcal{S}(x) - \mathcal{S}(y) + \Sigma(y) - \Sigma(x)$. Using the
previous paragraph,$ \mathcal{S}^{*}(x) - \mathcal{S}^{*}(y)$ is eventually of
magnitude $< 1$, $\mathcal{S}(x) - \mathcal{S}(y)$ is eventually $0$, and $\Sigma(y) - \Sigma(x)$
is eventually $-1$ or eventually $ 1$. We have reached the
required contradiction.
 
\end{proof}

\section{Proofs of the main results}\label{proofmainresults}
\subsection{Theorem \ref{T1}, \eqref{T1a}}\label{existchoice} ZFC proves that $EI$ has  $2^{\frak{c}}$ elements.
\begin{proof}
\begin{enumerate}

\item  Let $\tilde{L}$ be the set of the equivalence classes induced by \eqref{equivrel}.
  Consider the vector space $\tilde{V}$ generated by $\tilde{L}$. Let $\tilde{V}_1$ be the equivalence classes induced by \eqref{equivrel} of the $L^1$ functions in $C((0,1])$ and let $\tilde{B}_1$  be a Hamel  basis in $\tilde{V}_1$. By the usual construction using Zorn's Lemma let $\tilde{B}$ be a basis for $\tilde{L}$ containing $\tilde{B}_1$.

\item $0$ will represent the equivalence class of $0$. Note  that any representative of  $\tilde{V_1}$ is in $L^1$. 

\item The elements $b$ are linearly independent of each other. Indeed $\sum_{i\le N} c_i b_i=0$ implies
 $\sum_{i\le N} c_i \tilde{b}_i\sim_0 0$, a contradiction. 

\item Let $V$ be the vector space generated by the $b$'s and $V_1$ be the vector space generated by $b_{1}$'s (the representatives of $\tilde{b}_1\in \tilde{V}_1$). 

\item On $V_1$ we let $\Lambda$ be the linear functional $v_1\to \int_0^1 v_1$. We write $V=V_1\oplus V_2$; any $v$ can be uniquely written as $v=v_1+v_2, v_{1,2}\in V_{1,2}$. We let $\Lambda v=\Lambda v_1$. This is obviously a linear functional on $V$.

\item Let $f\in C((0,1])$. By assumption $f\in \tilde{f}\in \tilde{L}$ for some $\tilde{f}$, and $\tilde{f}$ can be written uniquely in the form $\tilde{f}=\sum_{i=1}^N c_i \tilde{b}_i$, which is equivalent to $f=\sum_{i=1}^N c_i {b}_i+h$, $h\sim_0 0$. The decomposition is unique since  $$\sum_{i=1}^N c_i {b}_i+h=0\Leftrightarrow \sum_{i=1}^N c_i \tilde{b}_i\sim_0 0\Leftrightarrow c_i=0\,\forall i\le N$$

\item Now we simply define
  \begin{equation}\label{Defint}
   \mathcal{P}_0 f= \int_1^x \sum_{i=1}^N c_i {b}_i+\Lambda\sum_{i=1}^N c_i {b}_i+\int_0^x h
  \end{equation}
where the last integral is the Lebesgue integral, which  exists since
$h\sim_0 0$.

\end{enumerate}
It is now straightforward to check that $\mathcal{P}_0$ is a  linear antiderivative  with the required properties. Eventual positivity comes from the fact that $\mathcal{P}_0$  coincides with $\int_0^x$ in $L^1$ and the fact that $\int_1^x f\to -\infty$ otherwise. The property $\mathcal{P}_0g'=g+const.$ follows from the fact that, for $x\in (0,1]$ we have  $(\mathcal{P}_0g')(x)=(\mathcal{P}_0g')(1)+\int_1^x g'$. 

We note that we have not used continuity in this proof, and the extension claimed after the statement of Theorem \ref{T1} is obvious.
\end{proof}

\subsection{Proof of Theorem \ref{T1} \ref{T1b}} \label{S72}  Existence results in $Op(f^{\Join}),f\in L^1$ are  immediate.
  ZFDC proves the following. $ Op(f^{\Join})$ has a  Borel measurable element
if and only if  $f\in L^1$. 
\begin{proof} This is immediate from Proposition \ref{PZFDC} and Theorem \ref{THEOREM 7.}.
\end{proof}
Note that the equivalence in Theorem \ref{T1} \eqref{T1b} does not
involve provability or definability notions. Arguably, any subset of or function between Polish spaces that is not Borel measurable, is mathematically pathological or at least mathematically undesirable. The Borel measurable sets form a natural
hierarchy of length the first uncountable ordinal, and it
can be further argued that any subset of or function that
does not lie in the first few levels of this hierarchy is
pathological or at least mathematically undesirable.
Borel measurability in Polish spaces is extensively
investigated in descriptive set theory, particularly in
connection with Borel equivalence relations and reductions
between them. See \cite{Hjorth}.
\subsection{Theorem \ref{T1} \ref{T1c}} 
ZFDC proves  that, if $f\notin L^1$ and  $Op(f^{\Join})\ne \emptyset$, then there
is a set of reals which is not Baire measurable.

\begin{proof}
  Assume that $ Op(f^{\Join}),f\notin L^1 $ is nonempty. By Proposition \ref{PZF}, let $\mathcal{S}\in Su$. By Theorem \ref{THEOREM 7.}, $\mathcal{S}$ is not Baire
measurable. Hence there is a subset of $\{0,1\}^{\NN}$ that is not
Baire measurable in $\{0,1\}^{\NN}$. Let $T \subseteq \{0,1\}^{\NN}$ consist of
removing the elements of $\{0,1\}^{\NN}$ that are eventually
constant. Then $T$ is homeomorphic to $\RR\setminus\QQ \subseteq \RR$. Also since we
have removed only countably many points from $\{0,1\}^{\NN}$, there
is a subset of $T$ that is not Baire measurable in $T$. Hence
there is a subset of $\RR\setminus\QQ$ that is not Baire measurable in
$\RR\setminus\QQ$. Hence there is a subset of $\RR$ that is not Baire
measurable in $\RR$. 
\end{proof}
\begin{Lemma}\label{738}
  ZFDC does not prove the existence of a set of
reals that is not Baire measurable.
\end{Lemma}
\begin{proof}
 This is proved in \cite{SolovayA}  assuming that ZFC + ``there exists a strongly inaccessible cardinal'' is consistent. It was subsequently proved in \cite{[Shxx]} assuming only that ZFC is consistent.
\end{proof}
\subsection{Theorem \ref{T1} \ref{T1d}}
 The statement $(\exists f\notin L^1)( Op(f^{\Join})\ne \emptyset)$ is not provable in ZFDC.
 \begin{proof}
   Suppose ZFDC proves $ Op(f^{\Join})  \ne\emptyset$ for an $f\notin L^1$. By Theorem \ref{T1} \eqref{T1c}, ZFDC
proves that there exists a set of reals that is not Baire
measurable. This contradicts Lemma \ref{738}.
 \end{proof} Theorem \ref{T14} follows from Theorem \ref{T1} \ref{T1a} and \ref{T1d}.
\subsection{Proof of Theorems \ref{New10} and \ref{New11}}\label{S73}
The most convincing negative results of this paper are
Theorems \ref{New10},  \ref{New11}, \ref{T2} and \ref{T2.5}. These involve
explicit definability. In many contexts in descriptive set
theory, we have non Borel measurability, yet we do have
demonstrably explicit definability. The most direct example
of this is by constructing an $A \subseteq \RR^2$ such that every Borel
measurable $B \subseteq \RR$ is of the form $\{y \in \RR: (c,y) \in A\}, c \in \RR$.
Then we can form the diagonal set $\{y \in \RR: (y,y) \notin A\}$, which
obviously differs from every Borel measurable $B \subseteq \RR$.
A more mathematically interesting example is as follows.
Consider the infinite product space $\QQ^\NN$, using the order
topology on $\QQ$. Then $\{x \in \QQ^\NN: \text{rng}(x)$ is a compact subset of
$\QQ\}$ is well known to be not Borel measurable.

Theorem \ref{New10} follows immediately from Theorem \ref{New11}.
\begin{Lemma}\label{Solovay}
Let $T$ be ZFC or ZFC extended by any standard
large cardinal hypothesis, such as on the Chart of
Cardinals in \cite{[Ka94]}.  Let $M$ be a countable model of $T$. There
is a countable mild forcing extension $M'$ of $M$ satisfying $T$
+ ``every set of reals is Baire measurable'' in which every
$M'$ definable set of reals of $M'$, with reals of $M'$ as
parameters, is internally Baire measurable.
\begin{proof}
  This result was proved in \cite{SolovayA} with ZFC replaced by
ZFC + ``there exists a strongly inaccessible cardinal''. This
result as stated was proved in the subsequent \cite{[Shxx]}.
\end{proof}
\end{Lemma}
 Here is the formal statement of Theorem \ref{New11}.
\newtheorem*{Theorem*}{Theorem \ref{New11} (formal)}
\begin{Theorem*}
  There is no formula $\varphi$  of ZFC with
exactly the free variables $x,y$, such that the following is
provable in ZFC. 
\begin{enumerate}[i]
\item \label{iti}$\varphi(x,y)\Rightarrow  x\in\RR\wedge (\exists! y)(\varphi(x,y))$. 
\item \label{itii} $(\exists x, y)(\varphi(x,y)\wedge (\exists f\notin L^1 \wedge y\in Op(f^{\Join})))$.
\end{enumerate}
This also holds for ZFC extended by any of the usual large
cardinal hypotheses, provided the extension results in a
consistent system.
\end{Theorem*}
\begin{proof}
Let $T$ be as in Lemma \ref{Solovay}. Let $\varphi$ be such that i,ii are
provable in $T$. Let $M,M'$ be as given by Lemma \ref{Solovay}.
By ii, choose $x,y \in M'$ such that  $\varphi(x,y) \wedge y \in  Op(y^{\Join}) \wedge
y \notin L^1$ holds in $M'$. By i, $ x \in  \mathbb{R}$ holds in $M'$, and $y$ is $M'$
definable from $x$. By Proposition \ref{providesum}, let $S \in  Su$, where $S$ is
$M'$ definable from $y$, and hence $M'$ definable from $x$. By
Theorem \ref{THEOREM 7.}, $S$ is
not Baire measurable in $M'$. By the explicit construction in
the proof
of Theorem \ref{T1} (c) that converts a non Baire measurable set
in $\{0,1\}^\NN$ to a non
Baire measurable set in $\mathbb{R}$, we obtain a set of reals,
internal to $M$, which is non
Baire measurable in the sense of $M'$, and also $M'$ definable
from a real internal to
$M'$. This contradicts Lemma \ref{Solovay}.
\end{proof}
\subsection{Proof of Theorem \ref{T2.5}}\label{S6.6-}
Once again, Theorem \ref{T2} follows immediately from Theorem \ref{T2.5}. Theorem \ref{T2.5} is formalized analogously to Theorem \ref{New11}.  We omit
the detailed formalization.
We now prove Theorem \ref{T2.5}, which we repeat here for the
convenience of the reader.

{\z \bf Theorem \ref{T2.5}}.
 Let $E\subset \ZZ^4$ be arithmetically presented, where ZFC proves that $E$ codes some  $f\in C((0,1])$. There is a definition which, provably in ZFC,
uniquely defines a function whose domain is a set of real numbers, with a value in  $Op(f^{\Join})$  if and only if ZFC proves $f\in L^1$. This also holds for
ZFC extended by the usual large cardinal hypotheses.
\begin{proof}
 Proof: Let $E$ be as given. 
Let $T$ be as in Lemma \ref{Solovay}. Let $T$ prove that $E$ codes $f \in
C((0,1])$. Suppose $T$ does not prove $f \in L^1$. Let $M$ be a
countable model of $T + f \notin L^1$. Let $M'$ be as given by Lemma
\ref{Solovay}. Then $M'$ also satisfies $f \notin L^1$. This is because $f \notin L^1$
is an arithmetic sentence.
Now suppose that $\varphi$ is a definition which, provably in $T$,
uniquely defines a function whose domain is a set of real
numbers, with a value in $Op(f^{\Join})$. Then in $M'$, we obtain an
element of $Op(f^{\Join})$ that is definable from an internal real
in $M'$.  Following the proof of Theorem \ref{New11}, we obtain a contradiction.
\end{proof}
To prove the weaker Theorems \ref{New10} and \ref{T2}, where there are no
real number parameters, it suffices to use Lemma \ref{Solovay}
with  $M$ definability without parameters. This is because the three spaces in question are explicitly defined. For
this weakened form of Lemma \ref{Solovay}, we can adhere to
\cite{SolovayA}, merely generically collapsing $\omega_1$ to $\omega$, and
weaken the assumption of the consistency of ZFC +
``there exists a strongly inaccessible cardinal'' to
the consistency of ZFC.

\section{Appendix: the Hadamard $p.f.$}\label{Appendix} In a nutshell, what is now called the Hadamard ``partie finie'' ($p.f.$,  finite part) relies on smoothness assumptions on $f$ to integrate by parts in \eqref{eq:defH2}. The infinite endpoint values are discarded at each step. This process lowers the order of the singularity of the integrand until it becomes $L^1$.  

As shown  by  M. Riesz, cf. \cite{Riesz}, a natural way to interpret $p.f.$  is through analytic continuation with respect to $\alpha$  of the right side of  \eqref{eq:defH2}, starting from a power of $s$ for which the integrand is in $L^1$. Analytic continuation from $\Re \alpha>0$ to $\Re \alpha>-n$ exists if \eqref{eq:sufc2} holds. This is manifestly so: indeed after integration by parts we obtain
\begin{equation}
  \label{eq:Riesz}
 \Gamma(\alpha)(J^\alpha f)(x)=\sum_{k=0}^{n-1}\frac{\Gamma(\alpha)f^{(k)}(x)}{\Gamma(\alpha+k+1)}x^{\alpha+k}+\int_0^xs^{\alpha+n-1}f^{(n)}(s)ds
\end{equation}

\bigskip 

For the distributional interpretation, see {\em e.g.} \cite{Hormander} \S 3.2.

This paper establishes that the above mentioned sufficient condition is also necessary in a deep sense:  there is no formula further extending \eqref{eq:Riesz} without smoothness assumption. In this sense, the Hadamard definition is optimal: it is necessary that $f$ have exactly the  regularity required by \eqref{eq:sufc2}, which is the same as the  regularity needed for the right side of \eqref{eq:Riesz} to make sense.
\subsection{The analytic case}\label{Ap}
Our main results can be adapted to study the $p.f.$ on meromorphic functions (where it exists) and on functions with essential singularities. 
The functions $\mathcal{A}$ we study are analytic in $D=\DD\setminus \{0\}$ and continuous up to the boundary. 
For analytic functions, the space $f^{\Join}$ cannot be adapted by simply replacing bounded $C^\infty$ functions with bounded analytic functions on $\DD\setminus \{0\}$, because zero would be a removable singularity, and any regularity at zero of $f$ would be inherited by $f^{\Join}$. Instead, we enlarge this space. We take a weight $0<w\in C((0,1]$, which we assume for simplicity to be w.t. $x^2 w(x)$ is decreasing, and define
$$w^{\Join}:=\{f\in\mathcal{A}:|f(z)|< w(|z|)\}=\mathcal{A}\cap L^\infty_w(D)$$
 We require (i)  $(\mathcal{P}_0(f))'=f$. However, in $\mathcal{A}$,  $\sim_0$ is trivial:  $f\sim_0 g\Rightarrow f\equiv g$. We impose on $\mathcal{P}_0$   a stronger requirement, one that is satisfied by $p.f.$   It is convenient to make the change of variables $t=1/x$, $x^2w(x)=W(t)$ and move the problem to infinity. We are then looking for an integral based at infinity. Note  that $W$ is increasing. We write $W(x)=x^{2g(x)-1}$; $g$ is also increasing. Evidently, if $g$ increases without bound, then $W$ is superpolynomial. Let $G=g^{-1}$ and assume without loss of generality that $g$ grows slowly enough that 
 \begin{equation}
   \label{eq:defbetak}
 \forall k\in \ZZ^+,\ \    \beta_k:=G(k)>5^k
 \end{equation}
Our condition (ii') is ``if the antiderivatives coincide on intervals, then they coincide'', more precisely
\begin{equation}
   \label{eq:conda} (ii')\ f_1\sim_{\infty} f_2 
\text{ iff } (\forall \,\,n\in  \ZZ^+  \(\int_{\beta_{n+1}}^{\beta_n} f_1= \int_{\beta_{n+1}}^{\beta_n} f_2\)\text{ then } \Bigg((\mathcal{P}_0 f_1)(\beta_1)=(\mathcal{P}_0 f_2)(\beta_1)\Bigg)
 \end{equation}
 \begin{Note}{\rm 
 Of course, if $f_1,f_2$ grow too slowly,    \eqref{eq:conda} may imply $f\equiv g$; this is not a problem, as there will always be suitable sequences  which do not entail $f\equiv g$. A serious problem however is that the equivalence relation depends on a  non canonical $(\beta_k)_{k\in \ZZ^+}$. We leave open the question of the existence of more natural and nontrivial equivalence relations on spaces with unique continuation.
} \end{Note}
For analytic functions with an isolated singularity at zero, the classical domain of $p.f.$ is optimal:
\begin{Theorem}
(A)  ZFC proves the existence of a  $\mathcal{P}_0$ on $D$ with the properties (i), (ii') and also: (iii) $\mathcal{P}_0$ is linear, and (iv) $\mathcal{P}_0 =p.f.$ on meromorphic functions. 

(B) ZFDC  proves the existence of a $\mathcal{P}_0$ on $w^{\Join}$ satisfying (i) and (ii') above, {\bf iff} $x^nw(x)\to 0$ as $x\to 0$ for some $n$. 
\end{Theorem}

\begin{proof} (A) is proved as in \S\ref{existchoice}, with small adaptations: $C((0,1])$ is replaced by  $D$,  ``$f\in L^1$'' by ``$f$ is meromorphic'' and $\int_0^1$ by $p.f.$ on $(0,1)$; we now require in (2) that, if $\tilde{V}_1$ has a meromorphic function inside, its representative  should be meromorphic. Finally we replace  $\int_0^x$ in \eqref{Defint} by $\int_{\beta_i}^x$ for any $\beta_i<x$ (cf. \eqref{eq:conda}).

For (B), if $w$ grows at most polynomially, $p.f.$ applies. In the opposite direction, we now construct, from a $\mathcal{P}_0$ based at infinity, an element of $Su$. 
\begin{Definition}\label{defg} {\rm 
As before $A=\{0,1\}^\NN$. Without loss of generality, we can restrict to the sequences with $a_1=0$. Let $s_k=\sum_{j=1}^{k}a_j,k\in \ZZ^+$. Consider the Cantor space of functions
\begin{equation}
  \label{eq:eqc1}
  \mathfrak{C}_1:=\left\{F_a: F_a(z)= \sum_{k=1}^{\infty}B_k\prod_{j\ne k}\(1-\frac{z}{\beta_j}\)^2,\,\ a\in A\right\};\  B_k:=\frac{s_k}{\prod_{j\ne k}\(1-\beta_k/\beta_j\)^2};
\end{equation}
where by construction $F_a(\beta_k)=s_k$. 
}\end{Definition}
\begin{Note}
 $F_a(\beta_k)$ will be our $\int_{1}^{\beta_{k}}f_a$, where  $f_a=F'_a$. 
\end{Note}

By \eqref{eq:defbetak} we have the following straightforward estimate \footnote{ For $x,y\ge 0$, the notation $x\lesssim y$ means $x\le C y$, where $C\ge 0$ does not depend on $x,y$, or relevant parameters. This notation is  standard.}. 
\begin{equation}
  \label{eq:eq34}
  |B_k|\lesssim s_k \beta_k^{-(2k-2)}\prod_{j=1}^{k-1}\beta_j^2 \le s_k\frac{\beta_{k-1}^2}{\beta_k^2}\le s_k5^{-2k}
\end{equation}
We first estimate the terms in the sum and the sum itself. By \eqref{eq:defbetak}, the sum $\sum_{k=1}^\infty |z|^2 |\beta_k|^{-2}$ converges and each infinite product in the sum converges (\cite{Lang}, see also the estimates below). It is clear that for a given $|z|$ all infinite products in \eqref{eq:eqc1}  are maximal when $z=-|z|$.

We have
\begin{equation}
  \label{eq:diform}
   \sum_{k=1}^{N}B_k\prod_{j\ne k}^{\infty}\(1+\frac{\rho}{\beta_j}\)^2= \prod_{j=1}^{\infty}\(1+\frac{\rho}{\beta_j}\)^2 \sum_{k=1}^{N}\frac{B_k}{(1+\rho/\beta_k)^2}\le  \prod_{j=1}^{\infty}\(1+\frac{\rho}{\beta_j}\)^2 \sum_{k=1}^{N}B_k
\end{equation}
which converges by  the assumption on $\beta_k$,  \eqref{eq:eq34}, and the convergence of $\sum_{j=1}^\infty |z|^2 |\beta_k|^{-2}$. We then also have
\begin{multline}
  \label{eq:prodest2}
   \sum_{k=1}^{\infty}B_k\prod_{j\ne k}^{\infty}\(1+\frac{\rho}{\beta_j}\)^2\le  \prod_{j=1}^{\infty}\(1+\frac{\rho}{\beta_j}\)^2 \sum_{k=1}^{\infty}B_k\\\lesssim  \prod_{j=1}^{\infty}\(1+\frac{\rho}{\beta_j}\)^2=\prod_{\beta_j<\rho}^{\infty}\(1+\frac{\rho}{\beta_j}\)^2\prod_{\beta_j\ge\rho}^{\infty}\(1+\frac{\rho}{\beta_j}\)^2\lesssim 4^M\rho^{2M}\le (4\rho^2)^{G^{-1}(\rho)}
\end{multline}
where $M$ is the largest $j$ s.t. $\beta_j<\rho$ for $j\le M$ and  we used $(1+x)^2<4x^2$ for $x>1$. The inequality implies, in particular that  {\em $f$ is entire.}

To estimate $f_a$, for $|z|\le \rho$ we simply use Cauchy's formula on a circle of radius $2\rho$:
\begin{equation}
  \label{eq:deri}
  |F_a'(z)|=\left|\frac{1}{2\pi i}\oint_{|s|=2\rho}\frac{F_a(s)}{(s-z)^2}\right|\lesssim\rho^{2g(\rho)-1}= W(\rho)\Rightarrow f_a(1/t)t^{-2}\in w^{\Join}
\end{equation}
Since $F_a(\beta_k)=s_k$, with $f_1(t)=f(1/t)t^{-2}$ and $\alpha_k:=1/\beta_k$ we get
    \begin{equation}
      \label{eq:choicealphak}
\forall k\in \ZZ^+ \Rightarrow \int_{\alpha_{k}}^{\alpha_{k+1}}f_1=(s_{k+1}-s_{k})=a_{k}
    \end{equation}
the desired result. 
\end{proof}

\section*{Acknowledgments} This research was partially supported by NSF DMS Grant 1108794 (OC), an Ohio State University  
Presidential Research Grant and by the John Templeton Foundation\footnote{The opinions expressed here are those of the author and do  
not necessarily reflect the views of the John Templeton Foundation} grant  
ID \#36297 (HF). The authors are grateful to A. Kechris for his very useful suggestions.


\begin{thebibliography}{99}
\bibitem{Marcone} A. Andretta, M.  Marcone, Definability in function spaces. Real Anal. Exchange 26, no. 1, 285--308 (2000/01). 
\bibitem{Sharpley} C. Bennett and R. Sharpley,   Banach Function Spaces , Elsevier,  Volume 129 (1988).
\bibitem{Charles}  C.E. Charles. The Baire category  
theorem implies the principle of dependent choices. Bull. Acad. Polon.  
Sci. Sér. Sci. Math. Astronom. Phys. 25 (1977), no. 10, 933--934.
\bibitem{Berry}  M. V. Berry and C. J. Howls, Hyperasymptotics, Proc. Roy. Soc. London Ser. A 430, no. 1880, pp. 653--668  (1990).
\bibitem{BlanchetFaye} L. Blanchet and G. Faye, J. Math. Phys, 41 7675--7714 (2000).
\bibitem{Calderon2} Calderon,  A. P.  and A, Zygmund, A. , "On the existence of certain singular integrals", Acta Mathematica 88 (1) pp:85–139 (1952).
\bibitem{CL} E.A. Coddington and N. Levinson, Theory of Ordinary Differential Equations, McGraw-Hill, New York, (1955).
\bibitem{Calderon6}  Calderon, A. P. and A, Zygmund, On Singular Integrals,  American Journal of Mathematics, Vol. 78, No. 2, pp. 289-309, Apr., 1956.
\bibitem{Conway} J.H. Conway, On numbers and games, A K Peters/CRC Press; 2nd edition (2000).
\bibitem{Book} O. Costin, Asymptotics and Borel summability, Chapmann \& Hall, New York (2009)
\bibitem{RDCostin} R.D. Costin,  Orthogonality of Jacobi and Laguerre polynomials for general parameters via the Hadamard finite part, J. Approx. Theory 162, no. 1, 141--152,  (2010).
\bibitem{Book} O. Costin {\it{Asymptotics and Borel Summability}}(CRC Press (2008)).
\bibitem{Costin} O. Costin, On Borel summation and Stokes phenomena for rank-1 nonlinear systems of ordinary differential equations,  Duke Math. J. 93, no. 2, pp. 289--344,(1998).
\bibitem{CostinT} O. Costin and S.  Tanveer, Nonlinear evolution PDEs in $\RR^+\times\CC^d$: existence and uniqueness of solutions, asymptotic and Borel summability properties. Ann. Inst. H. Poincar\'e Anal. Non Lin\'eaire 24, no. 5, 795–823,  (2007). 
\bibitem{Deligne} P. Deligne, Pierre, B.  Malgrange, J-P. Ramis, 
Singularit\'es irr\'eguli\`eres.  [Irregular singularities]
Correspondance et documents. [Correspondence and documents] Documents Math\'ematiques (Paris)  5. Soci\'et\'e Math\'ematique de France, Paris (2007).
\bibitem{Dynkin}  E.M. Dyn'kin, Pseudoanalytic extensions of smooth functions. The uniform
scale. AMS Transl. (2) 115 33--58 (1980). 
\bibitem{Ecalle1} J. \'Ecalle, Fonctions Resurgentes, Publ. Math. Orsay, vol. 81, Universit\'e de Paris-Sud, Depart\'ement de Math\'ematique, Orsay, (1981).
\bibitem{Jones} D.S. Jones,  Hadamard's finite part, Math. Methods Appl. Sci. 19  no. 13, pp. 1017--1052, (1996).
\bibitem{Ecalle2} J. \'Ecalle, Six lectures on transseries, analysable functions and the constructive proof of Dulac's conjecture, Bifurcations and Periodic Orbits of Vector Fields, NATO Adv. Sci. Inst. Ser. C Math. Phys. Sci., vol. 408, Kluwer, Dordrecht, (1993).
\bibitem{Hadamard} J. Hadamard,  Lectures on Cauchy's problem in linear partial differential equations,  New Haven Yale University Press, pp. 134-141 (1923).
\bibitem{Hjorth} G. Hjorth, Borel Equivalence Relations, 297--332, in: Handbook of  
Set Theory, Springer (2010).
\bibitem{Hormander} L. H\"ormander  The Analysis of Partial Differential Operators I, Springer, Berlin, (1983). 
\bibitem{Howls} C. J. Howls and  A. B.  Olde Daalhuis, Exponentially accurate solution tracking for nonlinear ODEs, the higher order Stokes phenomenon and double transseries resummation,  Nonlinearity 25 , no. 6, 1559--1584, (2012)
\bibitem{Hsiao} G.C. Hsiao and W. Wendland, 
Applied Mathematical Sciences Volume,  164,\S7-8 (2008).
\bibitem{[Je14]} T. Jech, Set Theory: The Third Millennium Edition, revised and  
expanded (Springer Monographs in Mathematics), (2014).
\bibitem{[Ka94]} A. Kanamori, The Higher Infinite, Perspectives in Mathematical  
Logic, Springer, (1994).
\bibitem{Katz} Y. Katznelson, An Introduction to Harmonic Analysis, Dover (1976).
\bibitem{Kechris} A.S. Kechris,  Classical Descriptive Set Theory, Springer (1995).
\bibitem{Kuijlars} A. Kuijlaars, A. Mart\'inez-Finkelshtein, R. Orive, Orthogonality of Jacobi polynomials with general parameters, Electronic Transactions on Numerical Analysis, 19 (2005), pp. 1--17
\bibitem{Krantz} S.G. Krantz, Geometric Function Theory: Explorations in Complex Analysis (Cornerstones), Birkh\"auser (2006).
\bibitem{Lang} S. Lang, Complex Analysis, fourth edition, (1999).
\bibitem{[LS67]} A. Levy and R. M. Solovay,
Measurable cardinals and the continuum hypothesis.
IJM 5 234--248 (1967).
\bibitem{Liouville} J.Liouville,  J. \'Ecole Polytechnique, 13(21),  163-86, (1832).

\bibitem{Musk} N. I. Muskhelishvili, Singular Integral Equations, Dover (2008).
\bibitem{Morton} R.D. Morton, A.M. Krall, Distributional weight functions for orthogonal polynomials, SIAM J. Math. Anal., 9 (1978), pp. 604--626.
\bibitem{Paycha} S. Paycha, 
Regularised Integrals, Sums and Traces: An Analytic Point of View (University Lecture Series) American Mathematical Society (2012).

\bibitem{Ramis} J. Martinet and 
J-P. Ramis,  Elementary acceleration and multisummability. I. Ann. Inst. H. Poincar\'e Phys. Th\'eor. 54, no. 4, 331–401  (1991). 
\bibitem{Riesz} M. Riesz, L'int\'egrale de Riemann-Liouville et le probl\`eme de Cauchy, Acta Mathematica  1--223, Marcel (1949). 
\bibitem{Riemann} B. Riemann, Versuch einer allgemeinen Auffassung der Integration und Differ-
entiation, Gesammelte Werke,  pp.62. (1876).
\bibitem{Rudin0} W. Rudin, Principles of mathematical analysis , McGraw-Hill (1976) pp. 75–78
\bibitem{Rudin} W. Rudin, Real and complex analysis.
\bibitem{Stein} E. Stein,  Singular integrals and differentiability properties of functions, Princeton Mathematical Series 30, Princeton, NJ: Princeton University Press, (1970).
\bibitem{SteinWeiss} E. Stein, G. Weiss,  Introduction to Fourier Analysis on Euclidean Spaces, Princeton University Press (1971).
\bibitem{Sellier} A. Sellier, Asymptotic Expansions of a Class of Integrals, Proc. R. Soc. Lond. A 445,  (1994).
\bibitem{Schwartz} L. Schwartz,  Th\'eorie des Distributions, Paris: Hermann \& Cie, (1950--1951).
\bibitem{SolovayA} R. M. Solovay, A model of set theory in which every set of
reals is Lebesgue measurable, Ann. of Math., 92, 1 - 56,
(1970). 
\bibitem{[Shxx]} S. Shelah, Saharon, "Can you take Solovay's
inaccessible away?", Israel Journal of Mathematics 48 (1):
1--47, (1984).
\bibitem{Suppes} P. Suppes,  Axiomatic Set Theory (1960),  Dover reprint (1972). 
\bibitem{Titchmarsh} E.Titchmarsh, Introduction to the theory of Fourier integrals, Oxford University: Clarendon Press (1986).



\end{thebibliography}
\end{document}